\documentclass[a4paper,oneside,10pt, reqno]{amsproc}
\usepackage{amsthm}
\usepackage{amsfonts}
\usepackage{amsmath}
\usepackage{amssymb}
\usepackage{mathrsfs}
\usepackage{epsfig}
\usepackage{graphicx}
\usepackage{epstopdf}
\usepackage{color}
\usepackage{ifthen}
\usepackage{float}

\makeatletter
\@namedef{subjclassname@2010}{%
\textup{2010} Mathematics Subject Classification}
\makeatother

\newtheorem{theorem}{Theorem}[section]
\newtheorem{proposition}[theorem]{Proposition}%[section]
\newtheorem{corollary}[theorem]{Corollary}%[section]
\newtheorem{lemma}[theorem]{Lemma}%[section]

\theoremstyle{remark}
\theoremstyle{definition}
\newtheorem{example}[theorem]{Example}

\newcommand{\bq}{\begin{equation}}
\newcommand{\eq}{\end{equation}}
\newcommand{\beqn}{\begin{eqnarray*}}
\newcommand{\eeqn}{\end{eqnarray*}}
\newcommand{\beq}{\begin{eqnarray}}
\newcommand{\eeq}{\end{eqnarray}}

\newcommand{\rar}{\rightarrow}

\newcommand{\bc}{\begin{centre}}
\newcommand{\ec}{\end{centre}}

\newcommand{\ba}{\begin{array}}
\newcommand{\ea}{\end{array}}

\newcommand{\inp}[2]{\langle{#1},\,{#2} \rangle}

\renewcommand{\Delta}{{\nabla}}

\newcommand*{\Ge}{\geqslant}

\newcommand*{\Le}{\leqslant}

\begin{document}
\title[Von Neumann's inequality for operator-valued multishifts]
{Von Neumann's inequality for commuting \\operator-valued multishifts}
\author[R. Gupta]{Rajeev Gupta}
\author[S. Kumar]{Surjit Kumar}
\author[S. Trivedi]{Shailesh Trivedi}
\address{Department of Mathematics and Statistics\\
Indian Institute of Technology  Kanpur, India}
   \email{rajeevg@iitk.ac.in}
\address{Department of Mathematics \\ Indian Institute of Science Bangalore, India}
\email{surjitkumar@iisc.ac.in}
\address{Department of Mathematics and Statistics\\
Indian Institute of Technology  Kanpur, India}
   \email{shailtr@iitk.ac.in}

\thanks{The work of first and second author was supported by Inspire Faculty Fellowship (Ref. No. DST/INSPIRE/04/2017/002367, DST/INSPIRE/04/2016/001008) while that of third author was supported by the National Post-doctoral Fellowship (Ref. No. PDF/2016/001681), SERB}

%   \thanks{The research of the
%third and fourth authors was supported by the NCN
%(National Science Center), decision No.
%DEC-2013/11/B/ST1/03613.}
   \subjclass[2010]{Primary 47B37, Secondary 47A13}
\keywords{operator-valued multishift, von Neumann's inequality, tensor product}

\date{}

\begin{abstract}
Recently, Hartz proved that every commuting contractive classical multishift with non-zero weights satisfies the matrix-version of von Neumann's inequality. We show that this result does not extend to the class of commuting operator-valued multishifts with invertible operator weights. In particular, we show that if $A$ and $B$ are commuting contractive $d$-tuples of operators such that $B$ satisfies the matrix-version of von Neumann's inequality and $(1, \ldots, 1)$ is in the algebraic spectrum of $B$, then the tensor product $A \otimes B$ satisfies the von Neumann's inequality if and only if $A$ satisfies the von Neumann's inequality. We also exhibit several families of operator-valued multishifts for which the von Neumann's inequality always holds. 
\end{abstract}

\maketitle

\section{Introduction}

The celebrated von Neumann's inequality \cite{von} says that if $T$ is a contraction on a Hilbert space $\mathcal H$, then $\|p(T)\| \Le \sup_{|z|<1} |p(z)|$ for every polynomial $p$. Generalizing this result, Sz.-Nagy \cite{Sz} proved that every contraction has a unitary dilation. Later Ando \cite{Ando} (see also \cite{SF}) extended this result and showed that every pair of commuting contractions dilates to a pair of commuting unitaries, and hence, every pair of commuting contractions satisfies the von Neumann's inequality. Thus it is natural to ask whether the von Neumann's inequality holds for a $d$-tuple of commuting contractions, $d \Ge 3$. Surprisingly, it fails for $d \Ge 3$. In fact, Varopoulos, in \cite{V1}, showed that there exists a big enough $d$ for which the von Neumann's inequality fails for a $d$-tuple of commuting contractions. In the addendum of the same paper, he together with Kaijser and independently Crabb and Davie \cite{CrD} gave examples of three commuting contractions which do not satisfy the von Neumann's inequality. Since then it has been one of the peculiar topics in operator theory.  In \cite[Question 36]{S}, Shields asked whether a $d$-tuple of commuting contractive weighted shifts (in other words, contractive classical multishift) satisfies the von Neumann's inequality. This question was attributed to Lubin and was explicitly mentioned in \cite{JL}. Recently, Hartz \cite{Hz} answered this question affirmatively and proved the following result:
\begin{theorem}\cite[Theorem 1.1]{Hz}\label{Hartz}
Let $T = (T_1, \ldots, T_d)$ be a contractive classical multishift with non-zero weights. Then $T$ dilates to a $d$-tuple of commuting unitaries.
\end{theorem}
In view of this, it is natural to ask whether the above result extends to the class of commuting operator-valued multishifts with invertible operator weights. The purpose of this note is to study the von Neumann's inequality for commuting operator-valued multishifts. A key tool in this study is the following  characterization for the tensor product of two $d$-tuples of commuting contractions to satisfy the von Neumann's inequality.

\begin{theorem}\label{intro-thm}
Let $d$ be a positive integer and let $A = (A_1, \ldots, A_d)$, $B = (B_1, \ldots, B_d)$ be two commuting $d$-tuples of contractions on the Hilbert spaces $\mathcal H$ and $\mathcal K$ respectively. Suppose that $B$ satisfies the matrix-version of von Neumann's inequality and $(1, \ldots, 1)$ belongs to the algebraic spectrum $\sigma(B)$ of $B$. Then $A \otimes B = (A_1 \otimes B_1, \ldots, A_d \otimes B_d)$ satisfies the von Neumann's inequality if and only if $A$ satisfies the von Neumann's inequality. 
\end{theorem}

Using this characterization, we prove that if $A = (A_1, \ldots, A_d)$ is a $d$-tuple of commuting contractions on a Hilbert space $H$ and $T = (T_1, \ldots, T_d)$ is a commuting operator-valued multishift on $\ell^2_H(\mathbb N^d)$ with operator weights given by $A^{(j)}_\alpha = A_j$ for all $\alpha \in \mathbb N^d$ and $j = 1, \ldots, d$, then $T$ satisfies the von Neumann's inequality if and only if $A$ satisfies the von Neumann's inequality. This readily yields a family of operator-valued multishifts with invertible operator weights which do not satisfy the von Neumann's inequality. This is in contrast with Theorem \ref{Hartz}. We conclude this paper with a concrete example of a commuting operator-valued multishift with invertible operator weights which does not satisfy the von Neumann's inequality. This example is motivated by the one which Kaijser and Varopoulos \cite{V1} gave to disprove the von Neumann's inequality for $3$-tuple of commuting contractions. It is worth mentioning that a commuting $d$-tuple of contractions dilates to a commuting $d$-tuple of unitaries if and only if it satisfies the matrix version of von Neumann's inequality \cite[Corollary 7.7]{P}. We refer the reader to \cite{P, AM, SF, MS, Sarkar} for recent developments related to the von Neumann's inequality.

We set below the notations used in posterior sections.
For a set $X$ and a positive integer $d$,  $X^d$ stands for the $d$-fold Cartesian product of $X$.
The symbols ${\mathbb N}, \ \mathbb Z,\ \mathbb R$ and $\mathbb C$ stand for the set of nonnegative
integers, set of integers, the field of real numbers and the field of complex numbers, respectively.
For $\alpha =
(\alpha_1, \ldots, \alpha_d) \in {\mathbb{N}}^d,$
we set $|\alpha|:=\sum_{j=1}^d
\alpha_j$.
For $w=(w_1, \ldots, w_d) \in \mathbb C^d$ and $\alpha=(\alpha_1, \ldots, \alpha_d) \in \mathbb N^d$, the complex conjugate $\overline{w} \in \mathbb C^d$ of $w$ is given by $(\overline{w}_1, \ldots, \overline{w}_d),$ while $w^\alpha$ denotes the complex number $\prod_{j=1}^d w^{\alpha_j}_j$. The symbol $\mathbb D^d$ is reserved for the open unit polydisc in $\mathbb C^d$ centered at the origin whereas $\overline{\mathbb D}^d$ stands for the closed unit polydisc in $\mathbb C^d$. 
Let $\mathcal H$ be a complex Hilbert space.
If $F$ is a subset of $\mathcal H$, the closed linear span of $F$ is denoted by $\bigvee \{x : x \in F\}$.  For a positive integer $m,$ the orthogonal direct
sum of $m$ copies of $\mathcal H$ is denoted by $\mathcal H^{(m)}.$ 
Let $\mathcal{B}({\mathcal H})$ denote the unital Banach algebra of
bounded linear operators on $\mathcal H.$ The multiplicative identity $I$ of 
$\mathcal{B}(\mathcal H)$ is sometimes denoted by $I_{\mathcal H}$.
The norm on  $\mathcal H$ is denoted by $\|\cdot\|_{\mathcal H}$ 
and whenever there is no confusion likely, 
we remove the subscript $\mathcal H$ from $\|\cdot\|_{\mathcal H}$. 
If $T \in \mathcal B(\mathcal H)$, then $T^*$ denotes the Hilbert space adjoint of $T$. An operator $T \in \mathcal B(\mathcal H)$ is said to be a {\it contraction} if $\|T\| \Le 1$. 
By a  {\it commuting $d$-tuple $T=(T_1, \ldots, T_d)$ in $\mathcal B(\mathcal H)$}, 
we mean a collection of commuting operators $T_1, \ldots, T_d$ in $\mathcal B(\mathcal H)$ and $T$ is said to be {\it contractive} if $T_j$ is a contraction for each $j=1, \ldots, d.$  
For $\alpha = (\alpha_1, \ldots, \alpha_d) \in \mathbb N^d$, we understand $T^{\alpha}$ as the operator $T^{\alpha_1}_1\cdots T^{\alpha_d}_d$, where we adhere to the convention that $A^0 = I_\mathcal H$ for $A \in  \mathcal B(\mathcal H)$. 
Let $T = (T_1, \ldots, T_d)$ be a commuting $d$-tuple in $\mathcal B(\mathcal H)$ and $\mathfrak B$ be the unital Banach algebra generated by $T_1, \ldots, T_d$. Then the algebraic spectrum $\sigma(T)$ of $T$ is given by
\beqn
\sigma(T) = \mathbb C^d \setminus \Big\{\lambda \in \mathbb C^d : \mbox{ there exist } S_1, \ldots, S_d \in \mathfrak B \mbox{ such that } \sum_{j=1}^d (T_j - \lambda_j)S_j = I_\mathcal H\Big\}.
\eeqn
The reader is referred to \cite{Cu} for a detailed account on various notions of spectra of a commuting tuple of operators.   
A $d$-tuple $T = (T_1, \ldots, T_d)$ of commuting contractions on $\mathcal H$ is said to satisfy the matrix-version of von Neumann's inequality if for every positive integer $m$,
\beqn
\|(p_{i, j}(T))_{_{1 \Le i, j \Le m}}\|_{\mathcal B(\mathcal H^{(m)})} \Le  \sup_{z \in \mathbb D^d}\|(p_{i, j}(z))_{_{1 \Le i, j \Le m}}\|_{\mathcal B(\mathbb C^m)}, \quad p_{i, j} \in \mathbb C[z_1, \ldots, z_d],
\eeqn
where $\mathbb C[z_1, \ldots, z_d]$ denotes the ring of polynomials over $\mathbb C$ in $d$ complex variables $z_1, \ldots, z_d$. 

\section{Von Neumann's inequality for tensor product of tuples}
 
In this section, we prove Theorem \ref{intro-thm}. We are grateful to the anonymous referee for his constructive comments which significantly improved the earlier version of this theorem. 
 
%\begin{lemma}\label{spectral-radius}
%Let $d$ be a positive integer and $A = (A_1, \ldots, A_d)$, $B = (B_1, \ldots, B_d)$ be two $d$-tuples of commuting contractions on the Hilbert spaces $\mathcal H$ and $\mathcal K$ respectively. If $(1, \ldots, 1)$ is in the algebraic spectrum $\sigma(B)$ of $B$, then for a polynomial $p \in \mathbb C[z_1, \ldots, z_d]$, 
%\end{lemma}
%
%\begin{proof}
%Let $p \in \mathbb C[z_1, \ldots, z_d]$ be a polynomial. Then by the spectral mapping theorem \cite{Cu}, we get $p(\sigma(B)) = \sigma(p(B))$. Hence
%\beqn
%r(p(B)) = \sup_{\lambda \in \sigma(p(B))} |\lambda| = \sup_{\lambda \in p(\sigma(B))} |\lambda| = \sup_{z \in \sigma(B)} |p(z)|.
%\eeqn
%Since $\mathbb T^d \subseteq \sigma(B) \subseteq \overline{\mathbb D}^d$, it follows that $\sup_{z \in \sigma(B)} |p(z)| = \sup_{z\in {\mathbb D}^d} |p(z)|$. This gives the desired conclusion.
%\end{proof}
%
%We are now ready to prove the Theorem \ref{intro-thm}.

\begin{proof}[Proof of Theorem \ref{intro-thm}]
Since $(1, \ldots, 1) \in \sigma(B)$, it follows from \cite[Proposition 1.2]{Cu} that there exists a multiplicative linear functional $\chi$ on the unital Banach algebra $\mathfrak B$ generated by $B_1, \ldots, B_d$ such that $\chi(B_j) = 1$ for all $j = 1, \ldots, d$.
Let $k\in\mathbb N$ and 
\beqn
p(z)=\sum_{\underset{|\alpha| \Le k}{\alpha \in \mathbb N^d}}a_{\alpha}z^\alpha, \quad a_\alpha\in\mathbb C,\ z \in \mathbb C^d,
\eeqn
be a polynomial. 
%Consider the matrix valued polynomial $p_A(z)$ given by 
%\beqn
%p_{_A}(z)=\sum_{\underset{|\alpha| \Le k}{\alpha \in \mathbb N^d}}a_{\alpha}A^\alpha z^\alpha, \ z \in \mathbb C^d.
%\eeqn
%Since $\mathcal B(\mathcal H) \otimes \mathcal B(\mathcal K)$ is isometrically embedded in $\mathcal B(\mathcal H \otimes \mathcal K)$, it follows that 
%\beq\label{p-tensor}
%p_{_A}(B) = p(A \otimes B). \textcolor{red}{??}
%\eeq
By Hahn-Banach theorem there exists a contractive linear functional $\phi$ on $\mathcal B(\mathcal H)$ such that $\phi(p(A)) = \|p(A)\|$. Then $\phi \otimes \chi$ is a contractive linear functional on $\mathcal B(\mathcal H) \otimes \mathfrak B$ (see \cite{B}). Hence, we get
\beqn
\|p(A \otimes B)\| \Ge |(\phi \otimes \chi) p(A \otimes B)| = \Big|\sum_{\underset{|\alpha| \Le k}{\alpha \in \mathbb N^d}}a_{\alpha}\phi(A^\alpha) \chi(B^\alpha) \Big| = |\phi(p(A))| = \|p(A)\|.
\eeqn
This establishes that $A$ satisfies the von Neumann's inequality if $A \otimes B$ satisfies the von Neumann's inequality. 

Conversely, assume that $A$ satisfies the von Neumann's inequality. Let $k\in\mathbb N$ and 
\beqn
p(z)=\sum_{\underset{|\alpha| \Le k}{\alpha \in \mathbb N^d}}a_{\alpha}z^\alpha, \quad a_\alpha\in\mathbb C,\ z \in \mathbb C^d,
\eeqn
be a polynomial.  Consider the polynomial $p(z,w)$ given by 
\beqn
p(z,w)=\sum_{\underset{|\alpha| \Le k}{\alpha \in \mathbb N^d}} a_{\alpha} z^\alpha w^\alpha,  \quad z, w \in \mathbb C^d.
\eeqn
Since $A$ satisfies the von Neumann's inequality, for each fixed $w \in \mathbb C^d$, we get
\beq\label{6}
\|p(A,w)\|=\Big\|\sum_{\underset{|\alpha| \Le k}{\alpha \in \mathbb N^d}}a_\alpha A^\alpha w^\alpha \Big\| \Le \sup_{z\in {\mathbb D}^d}|p(z,w)|.
\eeq
Let $\{e_\lambda : \lambda \in \Lambda\}$ be an orthonormal basis of $\mathcal H$. For a finite subset $F$ of $\Lambda$, let $P_F$ denote the orthogonal projection of $\mathcal H$ onto the subspace $\bigvee\{e_\lambda : \lambda \in F\}$ of $\mathcal H$. Define the (matrix-valued) polynomial
\beqn
p_{_F}(A, w) = \sum_{\underset{|\alpha| \Le k}{\alpha \in \mathbb N^d}}a_\alpha\, P_FA^\alpha  P_F \, w^\alpha, \quad w \in \mathbb C^d.
\eeqn 
Note that for each $w \in \mathbb C^d$, the net $p_{_F}(A, w)$ converges to $p(A, w)$ in the strong operator topology and $\|p_{_F}(A, w)\| \Le \|p(A, w)\|$. Since $B = (B_1, \ldots, B_d)$ satisfies the matrix version of von Neumann's inequality, it follows that
\beq\label{star}
\|p_{_F}(A, B)\| \Le \sup_{w \in \mathbb D^d} \|p_{_F}(A, w)\| \Le \sup_{w \in \mathbb D^d} \|p(A,w)\|.
\eeq
Also observe that the net $p_{_F}(A, B)$ converges to $p(A, B)$ in strong operator topology. Since $\mathcal B(\mathcal H) \otimes \mathcal B(\mathcal K)$ is isometrically embedded in $\mathcal B(\mathcal H \otimes \mathcal K)$, we identify $p(A, B)$ with $p(A \otimes B)$. Therefore, we get 
\beqn
\|p(A\otimes B)\| \overset{\eqref{star}}\Le \sup_{w \in \mathbb D^d} \|p(A,w)\|\overset{\eqref{6}}\Le \sup_{w\in {\mathbb D}^d} \sup_{z\in\mathbb D^d}|p(z,w)|= \sup_{z\in\mathbb D^d}|p(z)|.
\eeqn
%
%Observe that $p(A, B)=p(A \otimes B)$ and 
%since $B = (B_1, \ldots, B_d)$ satisfies matrix version of the von Neumann's inequality, it follows from \eqref{6} that 
%\beqn
%\|p(A \otimes B)\|_{\mathcal B(\mathbb C^n \otimes \mathcal H)}&=&\|p_{B}(A)\|_{\mathcal B(\mathbb C^n \otimes \mathcal H)} \Le \sup_{z\in\mathbb D^d}\|p_{z}(A)\|_{\mathcal B(\mathbb C^n)}\\
%&\Le& \sup_{z\in\mathbb D^d}\sup_{w\in {\mathbb D}^d}|p_{z}(w)|= \sup_{z\in\mathbb D^d}|p(z)|.
%\eeqn
This completes the proof of the theorem.
\end{proof}

\section{Operator-valued multishift and the von Neumann's inequality}

This section is devoted to the study of the von Neumann's inequality for commuting operator-valued multishifts. Before exhibiting a family of commuting contractive operator-valued multishifts with invertible operator weights which do not satisfy the von Neumann's inequality, we briefly recall the notion of operator-valued multishift. The notion of {\it operator-valued unilateral weighted shift} was introduced by Lambert in \cite{L} and was studied considerably thereafter (see \cite{LT, J} for related study). We refer to its several variable generalization as the {\it operator-valued multishift}. It seems that the notion of operator-valued multishift was not formally introduced and systematically studied earlier but it appeared at several places in the literature, see for instance \cite{C-S}, \cite{Pt}, \cite{CPT}. We now proceed towards the formal definition of operator-valued multishift.

Let $d$ be a positive integer and $\{H_\alpha : \alpha \in \mathbb N^d\}$ be a multisequence of complex Hilbert spaces. Let ${\mathcal H}=\oplus_{\alpha \in \mathbb N^d} H_\alpha$ be the orthogonal direct sum of $H_\alpha$, $\alpha \in \mathbb N^d$. Then $\mathcal H$ is a Hilbert space with respect to the following inner product:
\beqn \inp{x}{y}_{_{\mathcal H}}= \sum_{\alpha \in \mathbb N^d} \inp{x_{\alpha}}{y_{\alpha}}_{_{H_\alpha}},
\quad x=\oplus_{\alpha \in \mathbb N^d}x_{\alpha},\
y=\oplus_{\alpha \in \mathbb N^d} y_{\alpha} \in {\mathcal H}.\eeqn
If $H_\alpha = H$ for all $\alpha \in \mathbb N^d$, then we denote ${\mathcal H}=\oplus_{\alpha \in \mathbb N^d} H$ by $\ell^2_{H}(\mathbb N^d)$.
Let $\{{A^{(j)}_{\alpha}} : \alpha \in  \mathbb N^d,\ j = 1, \ldots, d \}$ be a multisequence of bounded linear operators $A^{(j)}_\alpha : H_\alpha \rar H_{\alpha+\varepsilon_j},$ where $\varepsilon_j$ is the $d$-tuple in $\mathbb N^d$ with $1$ in the $j^{\mbox{\tiny{th}}}$ place and zeros
elsewhere. 
An {\it operator-valued multishift} $T$ on $\mathcal H = \oplus_{\alpha \in \mathbb N^d} H_\alpha$ with operator weights $\{{A^{(j)}_{\alpha}} : \alpha \in  \mathbb N^d,\ j = 1, \ldots, d \}$ is a $d$-tuple of operators $T_1, \ldots, T_d$ in $\mathcal H$ defined by
\beqn 
\mathcal D(T_j) &:=& \Big \{ x=\oplus_{\alpha \in \mathbb N^d }x_{\alpha} \in {\mathcal H} :
\sum_{\alpha \in \mathbb N^d} \|A^{(j)}_{\alpha}x_{\alpha} \|^2 < \infty\Big\},\\
 T_j(\oplus_{\alpha \in \mathbb N^d}x_{\alpha})&:=& \oplus_{\alpha \in \mathbb N^d}
A^{(j)}_{\alpha-\varepsilon_j}x_{\alpha-\varepsilon_j}, \quad x=\oplus_{\alpha \in \mathbb N^d}x_{\alpha}
\in \mathcal D(T_j), \ j=1, \ldots, d.\eeqn 
For each $\alpha \in \mathbb N^d$ and $j=1,\ldots, d,$ if $\alpha_j=0$, then we
interpret $A^{(j)}_{\alpha-\varepsilon_j}$ to be a zero operator and $x_{\alpha-\varepsilon_j}$ as a zero vector.
Note that each $T_j$, $j=1, \ldots, d$, is a densely defined linear operator in $\mathcal H$.

The following proposition studies the basic properties of an operator-valued multishift. Its proof is straightforward and is left for the interested readers.  

\begin{proposition}
Let $d$ be a positive integer and $T = (T_1, \ldots, T_d)$ be an operator-valued multishift on $\mathcal H = \oplus_{\alpha \in \mathbb N^d} H_\alpha$ with operator weights $\{{A^{(j)}_{\alpha}} : \alpha \in  \mathbb N^d,\ j = 1, \ldots, d \}$. Then the following statements hold:
\begin{enumerate}
\item[(i)] For $j=1,\ldots,d$, $T_j$ is bounded if and only if
\beq \label{tj-bdd}
\sup_{\alpha \in \mathbb N^d} \|A^{(j)}_\alpha\| < \infty.
\eeq
%Further, if $T_j$ is bounded, then $\|T_j\| = \displaystyle\sup_{\alpha \in \mathbb N^d} \|A^{(j)}_\alpha\|$.
\item[(ii)] For $i,j=1, \ldots, d$, $T_i$ commutes with $T_j$ if and only if
\beq \label{commuting}
 A^{(i)}_{\alpha+\varepsilon_j} A^{(j)}_{\alpha}= A^{(j)}_{\alpha+\varepsilon_i} A^{(i)}_{\alpha} \ \mbox{for all } \alpha \in \mathbb N^d.
\eeq
\end{enumerate}
\end{proposition}

Let $T = (T_1, \ldots, T_d)$ be an operator-valued multishift on $\mathcal H = \oplus_{\alpha \in \mathbb N^d} H_\alpha$ with operator weights
$\{A^{(j)}_{\alpha}: \alpha \in \mathbb N^d,\  j=1, \ldots, d\}$. We refer to $T$ as {\it commuting operator-valued multishift} if the operator weights satisfy \eqref{tj-bdd} and \eqref{commuting}. Let us see how the class of classical multishifts is contained in that of operator-valued multishifts.

Let $\{w^{(j)}_\alpha : \alpha \in \mathbb N^d,\ j=1, \ldots,d\}$ be a multisequence of non-zero complex numbers such that $\sup_{\alpha \in \mathbb N^d} |w^{(j)}_\alpha| < \infty$ and $w^{(j)}_{\alpha+\varepsilon_i} w^{(i)}_\alpha = w^{(i)}_{\alpha+\varepsilon_j} w^{(j)}_\alpha$ for all $\alpha \in \mathbb N^d$, $i,j=1,\ldots,d$. Let $\mathcal H = \ell^2_{\mathbb C}(\mathbb N^d)$. Set $A^{(j)}_\alpha := w^{(j)}_\alpha I_{\mathbb C}$ for all $\alpha \in \mathbb N^d$ and $j=1,\ldots,d$. Then the commuting operator-valued multishift $T=(T_1, \ldots, T_d)$ with operator weights $\{A^{(j)}_{\alpha}: \alpha \in \mathbb N^d,\  j=1, \ldots, d\}$ is commonly known as {\it classical multishift} which was introduced in \cite{JL}. Note that $\ell^2_{\mathbb C}(\mathbb N^d)$ is nothing but $\ell^2(\mathbb N^d)$ and hence, in future, we shall write $\ell^2(\mathbb N^d)$ in place of $\ell^2_{\mathbb C}(\mathbb N^d)$.

We produce several families of commuting operator-valued multishifts for which the von Neumann's inequality always holds. We begin with the following lemma which generalizes \cite[Corollary 3.2]{L}.

\begin{lemma} \label{unit-wt-eq}
Let $d$ be a positive integer and $H$ be a complex Hilbert space. Let $T = (T_1, \ldots, T_d) \mbox { and } \tilde{T} = (\tilde{T_1}, \ldots, \tilde{T_d})$ be two commuting operator-valued multishifts on $\ell^2_{H}(\mathbb N^d)$ with respective unitary operator weights
$\{A^{(j)}_{\alpha}: \alpha \in \mathbb N^d, \  j=1, \ldots, d\}$ and $\{\tilde{A}^{(j)}_{\alpha}: \alpha \in \mathbb N^d, \  j=1, \ldots, d\}$. Then $T$ and $\tilde{T}$ are unitarily equivalent.
\end{lemma}

\begin{proof}
Let $U_0 = I_H$ and $U_{\alpha + \epsilon_j} = \tilde{A}^{(j)}_{\alpha} U_\alpha A^{(j)*}_{\alpha}$ for all $\alpha \in \mathbb N^d$ and $j = 1, \ldots, d$. We show that $U_\alpha$ is well-defined for each $\alpha \in \mathbb N^d$. To this end, first note that if $\alpha \in \mathbb N^d$ is such that $|\alpha| \Le 1$, then $U_\alpha$ is well defined. Now assume that   $|\alpha| \Ge 2.$ Let $\alpha = \beta + \epsilon_j = \gamma + \epsilon_k = \delta + \epsilon_j + \epsilon_k$ for some $\beta, \gamma, \delta \in \mathbb N^d$ and $j,k \in \{1, \ldots, d\}$. Using \eqref{commuting}, we get
\beqn
U_{\beta+\epsilon_j} &=&  \tilde{A}^{(j)}_{\beta} U_\beta A^{(j)*}_{\beta} = \tilde{A}^{(j)}_{\delta+\epsilon_k} U_{\delta+\epsilon_k} A^{(j)*}_{\delta+\epsilon_k} = \tilde{A}^{(j)}_{\delta+\epsilon_k} \tilde{A}^{(k)}_{\delta} U_\delta A^{(k)*}_{\delta} A^{(j)*}_{\delta+\epsilon_k}\\
&=& \tilde{A}^{(k)}_{\delta+\epsilon_j} \tilde{A}^{(j)}_{\delta} U_\delta A^{(j)*}_{\delta} A^{(k)*}_{\delta+\epsilon_j} = \tilde{A}^{(k)}_{\delta+\epsilon_j} U_{\delta+\epsilon_j} A^{(k)*}_{\delta+\epsilon_j} = \tilde{A}^{(k)}_{\gamma} U_\gamma A^{(k)*}_{\gamma}\\
&=& U_{\gamma+\epsilon_k}.
\eeqn
Since $\tilde{A}^{(j)}_{\alpha}$ and $A^{(j)}_{\alpha}$ are unitary operators,  it is clear that each $U_\alpha$ is a unitary operator on $H$. Set $U := \oplus_{\alpha \in \mathbb N^d} U_\alpha$. It is a routine verification to show that $U$ is a unitary on $\ell^2_{H}(\mathbb N^d)$ and $UT_jU^* = \tilde{T_j}$ for all $j=1, \ldots, d.$ This completes the proof.
\end{proof}

\begin{proposition}\label{unit-wt-eq-cor}
Let $d$ be a positive integer and $H$ be a complex Hilbert space. Let $T = (T_1, \ldots, T_d)$ be a commuting operator-valued multishift on $\ell^2_H(\mathbb N^d)$ with unitary operator weights $\{A^{(j)}_{\alpha} : \alpha \in \mathbb N^d,\  j=1, \ldots, d\}$. Then $T$ satisfies the von Neumann's inequality.
\end{proposition}

\begin{proof}
It follows from the preceding lemma that $T$ is unitarily equivalent to the operator-valued multishift on $\ell^2_H(\mathbb N^d)$ with operator weights being the identity operator on $H$. In other words, $T$ is unitarily equivalent to the unweighted multishift on $\ell^2_H(\mathbb N^d)$. Since the unweighted multishift on $\ell^2_H(\mathbb N^d)$ dilates to the (unweighted) bilateral multishift on $\ell^2_H(\mathbb Z^d)$, which is a $d$-tuple of commuting unitary operators, it follows that $T$ satisfies the von Neumann's inequality.
\end{proof}

Let us see another family of commuting operator-valued multishifts for which the von Neumann's inequality always holds.

\begin{proposition}\label{decom-classical}
Let $n,$ $d$ be positive integers and $T = (T_1, \ldots, T_d)$ be a commuting operator-valued multishift on $\ell^2_{\mathbb C^n}(\mathbb N^d)$ with operator weights given by 
\beqn
A^{(j)}_{\alpha} = \begin{pmatrix} w_{1,\alpha}^{(j)} & 0 & \ldots & 0 \\ 0 & w_{2,\alpha}^{(j)} & \ldots & 0\\
\vdots & \vdots & \ddots & \vdots\\
0 & 0& \ldots & w_{n,\alpha}^{(j)} \end{pmatrix} \mbox{ for all } \alpha \in \mathbb N^d \mbox{ and } j = 1, \ldots, d.
\eeqn
Then $T$ is unitarily equivalent to $W_1 \oplus \cdots \oplus W_n$, where $W_k$, $k = 1, \ldots, n$, is the classical multishift on $\ell^2(\mathbb N^d)$ with weights $\{w_{k,\alpha}^{(j)} : \alpha \in \mathbb N^d,\ j=1,\ldots, d\}$. 
\end{proposition}

\begin{proof}
Let $\{e_1, \ldots, e_n\}$ be the standard orthonormal set of vectors in $\mathbb C^n$. For $k=1, \ldots, n$, define 
\beqn
e_{\alpha, k} = \oplus_{\beta \in \mathbb N^d} x_\beta \mbox{ where } x_\beta = 0 \mbox{ if } \beta \neq \alpha \mbox{ and } x_\alpha = e_k.
\eeqn
For each $k = 1, \ldots n$, set 
\beqn
\mathcal M_k := \bigvee \big\{e_{\alpha, k} : \alpha \in \mathbb N^d \big\}. 
\eeqn
Then $\ell^2_{\mathbb C^n}(\mathbb N^d) = \mathcal M_1 \oplus \cdots \oplus \mathcal M_n$. Further, it is a routine verification to see that $\mathcal M_k$ is a reducing subspace of each $T_j$ and $(T_1|_{\mathcal M_k}, \ldots, T_d|_{\mathcal M_k})$ is unitarily equivalent to a classical multishift $W_k$ with weights $\{w_{k,\alpha}^{(j)} : \alpha \in \mathbb N^d,\ j=1,\ldots, d\}$ for $k=1, \ldots, n$. This completes the proof.
\end{proof}

As the von Neumann's inequality respects the direct sum, the following corollary immediately follows from Theorem \ref{Hartz} and the preceding proposition. 

\begin{corollary}
Let $n,$ $d$ be positive integers and $T = (T_1, \ldots, T_d)$ be a commuting contractive operator-valued multishift on $\ell^2_{\mathbb C^n}(\mathbb N^d)$ with operator weights being invertible $n \times n$ diagonal matrices. Then $T$ satisfies the von Neumann's inequality.
\end{corollary}

The following proposition facilitates us to produce a class of operator-valued multishifts for which the von Neumann's inequality does not hold.

\begin{proposition}\label{prop-for-c-example}
Let $H$ be a complex Hilbert space and $d$ be a positive integer. Suppose that $A = (A_1, \ldots, A_d)$ is a $d$-tuple of commuting contractions on $H$ and $T = (T_1, \ldots, T_d)$ is the commuting operator-valued multishift on $\ell^2_{H}(\mathbb N^d)$ with operator weights given by $A_\alpha^{(j)} = A_j$ for all $\alpha \in \mathbb N^d$ and $j=1, \ldots, d$. Then $T$ satisfies the von Neumann's inequality if and only if $A$ satisfies the von Neumann's inequality. 
\end{proposition}

\begin{proof}
Observe that the Hilbert space $\ell^2_{H}(\mathbb N^d)$ can be realized as $H \otimes \ell^2(\mathbb N^d)$. Hence it is not difficult to see that 
\beqn
T_j = A_j \otimes S_j \mbox{ for all } j = 1, \ldots, d,
\eeqn 
where $S = (S_1, \ldots, S_d)$ is the unweighted multishift on $\ell^2(\mathbb N^d)$. Note that $S$ is unitarily equivalent to the $d$-tuple of operators of multiplication by the coordinate functions on the Hardy space of the polydisc $\mathbb D^d$. Hence it follows that the Taylor spectrum of $S$ is $\overline{\mathbb D}^d$. Since the Taylor spectrum is contained in the algebraic spectrum, it follows that $(1, \ldots, 1)$ is in the algebraic spectrum of $S$. Now the desired conclusion is immediate from Theorem \ref{intro-thm}.
\end{proof}

\begin{corollary}
Let $d$ be a positive integer and $A = (A_1, \ldots, A_d)$ be a $d$-tuple of commuting $2 \times 2$ or $3 \times 3$ contractive matrices. Let $T = (T_1, \ldots, T_d)$ be the commuting operator-valued multishift on $\ell^2_{\mathbb C^n}(\mathbb N^d)$ $(n=2 \mbox{ or } 3)$ with operator weights given by $A_\alpha^{(j)} = A_j$ for all $\alpha \in \mathbb N^d$ and $j=1, \ldots, d$. Then $T$ satisfies the von Neumann's inequality.
\end{corollary}

\begin{proof}
The fact that any $d$-tuple of commuting $2 \times 2$ contractive matrices satisfies the von Neumann's inequality was established in \cite{D} while that for a $d$-tuple of commuting $3\times 3$ contractive matrices was proved in \cite{K}. Now rest of the proof is immediate from the preceding proposition.
\end{proof}

We are now ready to give the example which we mentioned in the beginning of this text. Our example is motivated from the one given by Kaijser and Varopoulos \cite{V1} to disprove the von Neumann's inequality for $3$-tuple of commuting contractions. Before this, observe that if a commuting $d$-tuple $V$ of contractive matrices is a counter-example to the von Neumann's inequality, then for some $t \in (0,1)$ the commuting $d$-tuple $A = tI + (1-t)V$ of invertible contractive matrices is also a counter-example to the von Neumann's inequality. Hence, if $S$ is the unweighted multishift on $\ell^2(\mathbb N^d)$, then the operator-valued multishift $A\otimes S$ satisfies all the conditions of Theorem \ref{intro-thm} and is the desired counter-example to the von Neumann's inequality.
The following example illustrates this discussion.
\begin{example}
Let $c\in (0,1/(6+\sqrt{30}))$. Following \cite[Definition 2.5]{GR} (see also \cite{V2}), consider the Varopoulos operators  on $\mathbb C^4$ given by
	\[V_j=
	\left(
	\begin{array}{cccc}
		0 & x_j & y_j& 0\\
		0 & 0 & 0 & x_j\\
		0 & 0 & 0 & y_j\\
		0 & 0 & 0 & 0\\
	\end{array}
	\right), \quad  j=1,2,3,\]
	where $x_j,y_j\in\mathbb R$ and $x_j^2+y_j^2=(1-c)^2$ for each $j=1,2,3.$ Set $A_j:=cI+V_j$ for all $j=1,2,3.$ Since $(V_1,V_2,V_3)$ is a commuting tuple, it follows that $A=(A_1,A_2,A_3)$ is also a commuting $3$-tuple of invertible matrices. Moreover, $\|A_j\|\leq c+\|V_j\|=1$ for all $j=1,2,3.$ 
	Let $T = (T_1, T_2, T_3)$ be the commuting operator-valued multishift on $\ell^2_{\mathbb C^4}(\mathbb N^d)$ with operator weights given by $A_\alpha^{(j)} = A_j$ for all $\alpha \in \mathbb N^d$ and $j=1, 2, 3$. 
	Now, consider the Varopoulos-Kaijser polynomial \cite{V1} (see also \cite{Gu})
	$$p_{\!_V}(z_1,z_2,z_3):=z_{1}^{2}+z_{2}^{2}+z_{3}^{2}-2z_1z_2-2z_2z_3-2z_3z_1.$$
	It is shown in \cite{V1} (see also \cite{Hol}, \cite{GR}) that $\sup_{z\in\mathbb D^3}|p_{\!_V}(z)|=5.$ Further, we observe that $p_{\!_V}(A_1,A_2,A_3)$ is given by
\beqn
	\left(
	\begin{smallmatrix}
	\displaystyle\sum_{j,k=1}^{3}c^2 a_{jk} &\displaystyle \sum_{j,k=1}^{3}c\, a_{jk}(x_j+x_k) & \displaystyle\sum_{j,k=1}^{3}c\, a_{jk}(y_j+y_k) & \displaystyle\sum_{j,k=1}^{3}a_{jk}\langle X_j,X_k\rangle\\
	0 & \displaystyle\sum_{j,k=1}^{3}c^2 a_{jk}  & 0 &  \displaystyle \sum_{j,k=1}^{3}c\, a_{jk}(x_j+x_k)\\
	0 & 0 &   \displaystyle\sum_{j,k=1}^{3}c^2 a_{jk}  & \displaystyle\sum_{j,k=1}^{3}c\, a_{jk}(y_j+y_k)\\
	0 & 0 & 0 & \displaystyle\sum_{j,k=1}^{3}c^2 a_{jk}
	\end{smallmatrix}
	\right),
	\eeqn	
	where \begin{equation*}
	(a_{jk}):=
	\left(
	\begin{array}{rrr}
		1 & -1 & -1\\
		-1 & 1 & -1\\
		-1 & -1 & 1\\
	\end{array}
	\right),
\end{equation*}
and $X_j=(x_j,y_j)$ for $j=1,2,3.$
	From above it can be concluded that 
\beqn
\|p_{\!_V}(A_1,A_2,A_3)\|&\Ge& \Big|\sum_{j,k=1}^{3}a_{jk}\langle X_j,X_k\rangle\Big|\\
&=& 3(1-c)^2-2(\langle X_1,X_2\rangle + \langle X_2,X_3\rangle + \langle X_3,X_1\rangle).
\eeqn	
Following \cite[Lemma 2.18]{Thesis}, it can be shown that the right hand side of the above inequality achieves its maximum value at $X_1=(1-c)(1,0),\, X_2=(1-c)(-1/2,\sqrt{3}/2)$ and $X_3=(1-c)(-1/2,-\sqrt{3}/2)$ and therefore
\beqn
\|p_{\!_V}(A_1,A_2,A_3)\|\Ge 6(1-c)^2.
\eeqn
Thus using this and the facts that $c<1/(6+\sqrt{30})$ and $\sup_{z\in\mathbb D^3}|p_{\!_V}(z)|=5$, we conclude that $\|p_{\!_V}(A_1,A_2,A_3)\|>\sup_{z\in\mathbb D^3}|p_{\!_V}(z)|$. Now from Proposition \ref{prop-for-c-example}, we deduce that the operator-valued multishift $T=(T_1,T_2,T_3)$ does not satisfy the von Neumann's inequality.\end{example}

\medskip \textit{Acknowledgment}.
The authors are grateful to Gadadhar Misra and Sameer Chavan for several helpful suggestions and constant support. We express our gratitude to the faculty and the administration of Department of Mathematics and Statistics, IIT Kanpur and Department of Mathematics, IISc Bangalore for their warm hospitality during the preparation of this paper. We sincerely thank the referee for several constructive comments which considerably improved the presentation of this paper.

\end{document}